\documentclass[12pt]{amsart}
\usepackage{}
\usepackage{geometry}
\usepackage{amsmath, amsthm, amssymb}
\newtheorem{thm}{Theorem}[section]
\newtheorem{cor}[thm]{Corollary}
\newtheorem{prop}[thm]{Proposition}
\newtheorem{lem}[thm]{Lemma}
\newtheorem{defn}[thm]{Definition}
\newtheorem{rem}[thm]{Remark}

\title[Multiple averages for certain averages along shifted primes]{Multiple recurrence and convergence for certain averages along shifted primes}
\author{Wenbo Sun}
\address{Department of Mathematics, Northwestern University, 2033 Sheridan Road Evanston, IL 60208-2730, USA}
\email{swenbo@math.northwestern.edu}

\begin{document}

\maketitle

\begin{abstract}
We show that any subset $A\subset\mathbb{N}$ with positive upper Banach density contains the pattern
$\{m,m+[n\alpha],\dots,m+k[n\alpha]\}$,
for some $m\in\mathbb{N}$ and $n=p-1$ for some prime $p$, where $\alpha\in\mathbb{R}\backslash\mathbb{Q}$. Making use for the Furstenberg Correspondence Principle, we do this by proving an associated recurrence result in ergodic theory along the shifted primes. We also prove the convergence result for the associated averages along primes and indicate other applications of these methods.
\end{abstract}

\section{Statement of Results}
\subsection{Introduction}
A \emph{measure preserving system} $(X,\mathcal{X},\mu, T)$ consists of a probability space $(X,\mathcal{X},\mu)$ and a measurable, measure preserving
transformation $T$ acting on it. Throughout this paper, we assume $T$ is invertible.
The limit behavior (existence and positivity) in $L^{2}(\mu)$ of a measure preserving system $(X,\mathcal{X},\mu, T)$ of multiple averages of the
form
\begin{equation}\label{form}\nonumber
\begin{split}
\lim_{N\rightarrow\infty}\frac{1}{N}\sum_{n=1}^{N}f_{1}(T^{a_{1}(n)}x)\cdot\ldots\cdot f_{k}(T^{a_{k}(n)}x)
\end{split}
\end{equation}
 has been widely studied in recent years for different choices of functions $a_{i}(n)$, where
$a_{i}(n)\colon\mathbb{N}\rightarrow\mathbb{N}$ are integer sequences (see, for example,
~\cite{B}, ~\cite{Fu}, ~\cite{FW}, ~\cite{HKp}, ~\cite{12},\\ ~\cite{L}, ~\cite{Z}).

In this paper, we study similar results
for averages over shifted primes, meaning, the primes or the primes plus or
minus 1. Let $\mathbb{P}$ denote the set of all primes and $\pi(N)$ denote the number of primes up to $N$. For all $x\in\mathbb{R}$, $\{x\}$ denotes the fractional part of $x$ and $[x]$ denotes the largest integer which is no larger than $x$. We prove the following theorem:
\begin{thm}\label{alpha2}Let $k\in\mathbb{N}, \alpha\in\mathbb{R}\backslash\mathbb{Q}$, $(X,\mathcal{X},\mu, T)$ be a measure preserving system. Let $f_{1},\dots,f_{k}\in L^{\infty}(\mu)$. Then the average
\begin{equation}\nonumber
\begin{split}
\frac{1}{\pi(N)}\sum_{p\in [N]\cap\mathbb{P}}T^{[p\alpha]}f_{1}\cdot\ldots\cdot T^{k[p\alpha]}f_{k}
\end{split}
\end{equation}
converges in $L^{2}(\mu)$ as $N\rightarrow\infty$. The conclusion is also true if $i[n\alpha]$ is replaced by $[in\alpha]$ for $i=1,\dots,k$.
\end{thm}

Frantzikinakis, Host and Kra ~\cite{3} first proved the linear polynomial case of Theorem \ref{alpha2} (i.e. the case all $i[n\alpha]$ are replaced with $in$), with the proof of the case $k\geq 3$ conditional upon the results of ~\cite{7} and ~\cite{6} that were subsequently proven. Then Wooley and Ziegler ~\cite{WZ} proved the convergence result for the polynomial case (i.e. the case all $i[n\alpha]$ are replaced with some polynomial with integer coefficients) for a single transformation. After that,
these results were generalized to the multi-dimensional polynomial case for commutative transformations by Frantzikinakis, Host and Kra ~\cite{4}. It is natural to ask whether one can prove a version of these theorems for generalized polynomials. Theorem \ref{alpha2} extends these results to linear generalized polynomials.

By what is now referred to as the Furstenberg Correspondence Principle, first introduced in ~\cite{Fu}, the positivity of
\begin{equation}\nonumber
\begin{split}
\int f\cdot T^{a_{1}(n)}f\cdot\ldots\cdot T^{a_{k}(n)}f d\mu
\end{split}
\end{equation}
 for some measure preserving system $(X,\mathcal{X},\mu, T)$ and some function $f\geq 0$ not identically 0 is equivalent to the positivity of the
upper Banach density $d^{*}$ of the set
\begin{equation}\nonumber
\begin{split}
d^{*}\Bigl(E\cap(E-a_{1}(n))\cap\dots\cap(E-a_{k}(n))\Bigr)
\end{split}
\end{equation}
for some $E\subset\mathbb{N}$, where $d^{*}(E)=\limsup_{\vert I\vert\rightarrow\infty}\frac{\vert I\cap E\vert}{\vert I\vert}$. So it is natural to ask when the limit in Theorem \ref{alpha2} is positive. We prove the following theorem:
\begin{thm}\label{alpha} Let $(X,\mathcal{X},\mu, T)$ be a measure preserving system, $A$ be a measurable set with $\mu(A)>0$ and
$\alpha\in\mathbb{R}\backslash\mathbb{Q}$ be an irrational number. Then the set of integers $n$ such that
\begin{equation}\nonumber
\begin{split}
\mu(A\cap T^{-[n\alpha]}A\cap\dots\cap T^{-[kn\alpha]}A)>0
\end{split}
\end{equation}
has nonempty intersection with $\mathbb{P}+1$, and with $\mathbb{P}-1$. The conclusion is also true if $[in\alpha]$ is replaced by $i[n\alpha]$ for $i=1,\dots,k$.
\end{thm}

We immediately deduce the following:
\begin{cor}\label{Fualpha} For any $k\in\mathbb{N}, \alpha\in\mathbb{R}\backslash\mathbb{Q}$, every subset $A\subset\mathbb{N}$ with positive upper Banach density contains the following pattern:
\begin{equation}\nonumber
\begin{split}
\{m,m+[n\alpha],\dots,m+[kn\alpha]\},
\end{split}
\end{equation}
for some $m\in\mathbb{N}, n\in\mathbb{P}+1$, and for some $m\in\mathbb{N}, n\in\mathbb{P}-1$. The conclusion is also true if $[in\alpha]$ is replaced by $i[n\alpha]$ for $i=1,\dots,k$.
\end{cor}

\begin{rem}
It is worth noting that Theorem \ref{alpha} fails if $\mathbb{P}+1$ is replaced with $\mathbb{P}$. For example, if $(X,\mathcal{X},\mu, T)$ is the 1-dimensional torus with $Tx=x+\frac{1}{2\sqrt{2}} (\mod 1)$ and $A=(-1/8,1/8)$, then
$\mu(A\cap T^{-[n\sqrt{2}]}A)>0$
only if $n$ is even.
\end{rem}

The first result in this direction was obtained by
 Sark\"{o}zy ~\cite{14}, who proved that the difference set of $E-E$ for a set $E$ with positive upper Banach density contains a shifted prime $p-1$
 for some $p\in\mathbb{P}$ (and a similar result holds for $p+1$). The linear polynomial case of Theorem \ref{alpha} was proved in ~\cite{3}, with the proof of the case $k\geq 3$ conditional upon the results of ~\cite{7} and ~\cite{6} that were subsequently proven. Then Wooley and Ziegler ~\cite{WZ} proved the recurrence result for the polynomial case for a single transformation. Then Bergelson, Leibman and Ziegler ~\cite{short} proved the linear polynomial case of Theorem \ref{alpha} for the multi-dimensional case.
The multi-dimensional polynomial version of Theorem \ref{alpha} for commutative transformations was proved by Frantzikinakis, Host and Kra ~\cite{4}.


The method used in proving Theorems \ref{alpha2} and \ref{alpha} also applies to other results on shifted
primes for cubes and certain weighted averages. We indicate how such results can be obtained in Section 5.
\subsection{Strategy and Organization}
The technique used in ~\cite{4} was to compare the average under consideration along the shifted primes to the analogous
average along all the integers. In order to obtain results for averages along $[n\alpha]$ (Theorems \ref{alpha2} and \ref{alpha}), we need to modify the technique. Roughly speaking, the difficulty is that the difference $a(n+h)-a(n)$ is independent of $n$ when $a(n)=cn, c\in\mathbb{N}$, but this is not the case
when $a(n)=[n\alpha]$.
To overcome this obstacle, we need to consider only those $n$ such that $\{n\alpha\}$ lies in some short interval to guarantee $a(n+h)-a(n)$ is constant.

The background material is presented in Section 2. Section 3 is used to obtain some estimate for Gowers norms that is used in later sections. We present
the proofs of Theorems \ref{alpha2} and \ref{alpha} in Section 4. Other applications of the estimate for Gowers norms are presented in Section 5.
\\ \textbf{Ackonwledgment.} The author would like to thank N. Frantzikinakis for helpful discussions related to the estimate for Gowers norms and the material in Section 4, B. Kra for all the instructive advices in preparation of this article, and the referee for the patient reading and helpful suggestions.
\section{Background}
\subsection{Notation}
Throughout this paper, we denote the set of positive integers by
$\mathbb{N}$, the set of real numbers by $\mathbb{R}$, and the set of rational numbers by $\mathbb{Q}$. Write $\mathbb{Z}_{N}=\mathbb{Z}/{N\mathbb{Z}}$. When needed, the set $\mathbb{Z}_{N}$ is identified with
$[N]\colon=\{1,\dots,N\}$.

We use $o_{N}(1)$ to denote a quantity that
converges to 0 when $N\rightarrow\infty$ and all other parameters are fixed. The expression $a(n)\ll b(n)$ stands for $a(n)\leq Cb(n)$ for some constant $C$. If $C$ depends on some parameter $d$, we write $a(n)\ll_{d}b(n)$.

If $A$ is a subset of a space $X$, we write $\bold{1}_{A}(x)$ to be the index function of $A$, taking value 1 for $x\in A$ and 0 elsewhere. We denote the space of continuous functions on $X$ by $C(X)$. For any complex-valued function $g(n)$, denote
$\mathcal{C}g(n)=\overline{g(n)}$. If $S$ is a finite set and
$a\colon  S\rightarrow\mathbb{C}$ is a function on $S$, we write $\mathbb{E}_{n\in S}a(n)=\frac{1}{|S|}\Sigma_{n\in S}a(n)$.

For any $t\in\mathbb{N}$ and any element $\epsilon=(\epsilon_{1},\dots,\epsilon_{t})\in\{0,1\}^{t}$, denote $\vert\epsilon\vert=\sum_{i=1}^{t}\epsilon_{i}$.

\subsection{Gowers norms}

For any function $a\colon  \mathbb{Z}_{N}\rightarrow\mathbb{C}$, we inductively define:

\begin{equation}\nonumber
\begin{split}
\Vert a\Vert_{U^{1}(\mathbb{Z}_{N})}=\bigl\vert\mathbb{E}_{n\in\mathbb{Z}_{N}}a(n)\bigr\vert
\end{split}
\end{equation}
and

\begin{equation}\nonumber
\begin{split}
\Vert a\Vert_{U^{d+1}(\mathbb{Z}_{N})}=\Bigl(\mathbb{E}_{h\in\mathbb{Z}_{N}}\bigl\Vert
a_{h}\overline{a}\bigr\Vert_{U^{d}(\mathbb{Z}_{N})}^{2^{d}}\Bigr)^{1/{2^{d+1}}},
\end{split}
\end{equation}
where $a_{h}(n)=a(n+h)$. Gowers ~\cite{G} showed that this defines a norm on functions on $\mathbb{Z}_{N}$ for $d\geq 1$. These norms were later used by Green, Tao, Ziegler and others in studying the primes (see, for example, ~\cite{5},~\cite{65} and ~\cite{6}). Analogous semi-norms were defined in the ergodic setting by Host and Kra ~\cite{12}.

\subsection{Nilsequences and Nilmanifolds}

Let $G$ be a Lie group. Denote $G_{1}=G,$ $G_{i+1}=[G, G_{i}]=\{ghg^{-1}h^{-1}\colon g\in G, h\in G_{i}\}, i\geq 1$. We say that $G$ is \emph{d-step nilpotent} if $G_{d+1}=\{1\}$ and the system $(G/\Gamma, \mathcal{X}, \mu, T)$ is called a \emph{$d$-step nilsystem}, where $\Gamma\subset G$
is a discrete cocompact subgroup, $\mu$ is the measure
induced by the Haar measure of $G$, $\mathcal{X}$ is the Borel $\sigma$-algebra and $T$ is the transformation given by $T(x)=gx$ for some fixed $g\in G$. If $F$ is a continuous function
on a d-step system $(G/\Gamma, \mathcal{X}, \mu, T)$, for any $x\in G/\Gamma$, we say that $\{F(T^{n}x)\}_{n\in\mathbb{N}}$ is a \emph{basic d-step nilsequence}. A sequence is called a \emph{d-step nilsequence} if it is the uniform limit of basic d-step nilsequences.

Nilsequences were introduced by Bergelson, Host and Kra ~\cite{BHK} and they play an important role in the study of multiple averages and in the estimate of Gowers norms (see, for example, ~\cite{BHK}, ~\cite{5},~\cite{65},~\cite{6}~\cite{HKp},~\cite{12},~\cite{8},~\cite{L} and ~\cite{Z}).
As we need to be quantitative regarding these nilmanifolds, we endow each manifold with an arbitrary smooth Riemannian metric. We then define the
\emph{Lipschitz constant} of a basic nilsequence $\{F(T^{n}x)\}_{n\in\mathbb{N}}$ to be the Lipschitz constant of $F$. Notice that the Lipschitz constant of a basic
nilsequence $\{F(T^{n}x)\}_{n\in\mathbb{N}}$ is independent of the transformation $T$.
\subsection{van der Corput Lemma}
The use of the van der Corput Lemma in studying
multiple averages in ergodic theory was introduced by Bergelson ~\cite{B}. In this paper, we use a variation of the estimation of van der Corput:
\begin{lem}\label{van} Let $\{v(n)\}_{n=1}^{N}$ be a sequence of elements in a Hilbert space with norm $\Vert\cdot\Vert$ and inner product $\langle\cdot, \cdot\rangle$. Then
\begin{equation}\nonumber
\begin{split}
\Bigl\Vert\frac{1}{N}\sum_{n=1}^{N}v(n)\Bigr\Vert^{2}\ll\frac{1}{N^{2}}\sum_{n=1}^{N}\Bigl\Vert v(n)\Bigr\Vert^{2}
+\frac{1}{N}\sum_{h=1}^{N}\Bigl\vert\frac{1}{N}\sum_{n=1}^{N-h}\langle v(n+h), v(n)\rangle\Bigr\vert.
\end{split}
\end{equation}
\end{lem}
The proof of a special case of this lemma can be found in ~\cite{9} and the proof of the general
case is essentially the same.

\section{Estimation of Modified von Mangoldt Function}
Throughout this paper, we let $\mathbb{T}=\mathbb{R}/\mathbb{Z}$ denote the 1-dimensional torus with Haar measure $\lambda$, associated $\sigma$-algebra $\mathcal{X}$
and transformation $R_{\alpha}(x)=x+\alpha ($mod $1)$, and we call the system $(\mathbb{T},\mathcal{X},\lambda,R_{\alpha})$ the 1-dimensional torus. For convenience, we sometimes identify $\mathbb{T}$ with the interval $[0,1]$.

Let $\Lambda\colon  \mathbb{N}\rightarrow\mathbb{R}$ denote the von Mangoldt function, taking the value $\log{p}$ on the prime $p$ and its powers and 0
elsewhere, and let $\Lambda'(n)=\bold{1}_{\mathbb{P}}(n)\Lambda(n)$. We use the von Mangoldt function to replace the indicator function $\bold{1}_{\mathbb{P}}(n)$ since it has better analytic properties. We have the following lemma:

\begin{lem}\label{lambda2}
If a function $a\colon \mathbb{N}^{k}\rightarrow\mathbb{C}$ is bounded, then
\begin{equation}\nonumber
\begin{split}
\lim_{N\rightarrow\infty}\Bigl\vert{\frac{1}{\pi(N)^{k}}\sum_{p_{i}\in\mathbb{P}, p_{i}\leq N}{a(p_{1},\dots,p_{k})}}-{\frac{1}{N^{k}}\sum_{1\leq
n_{1},\dots,n_{k}\leq N}{\prod_{i=1}^{k}\Lambda'(n_{i})a(n_{1},\dots,n_{k})}}\Bigr\vert=0.
\end{split}
\end{equation}
\end{lem}

The proof of the case $k=1$ can be found in for example ~\cite{3}, and the general case can be derived easily from the special case.

Throughout this paper we denote $W=\Pi_{p\in\mathbb{P}, p<w}{p}$ for $w>2$, i.e. $W$ is always assumed to be an integer depending on $w$. For $r\in\mathbb{N}$, we always denote
$\Lambda'_{w, r}(n)=\frac{\phi(W)}{W}\Lambda'(Wn+r)$,
where $\phi(n)=\sum_{1\leq d\leq n,d\nmid n}1$ is the Euler function, and $\Lambda'_{w, r}(n)$ is referred to as a \emph{modified von Mangoldt function}.
The key to the study of convergence and recurrence results along primes is the estimate of the Gowers norms of this modified von Mangoldt function.
The purpose for this section is to establish a variation of Theorem 7.2 of ~\cite{5}, which allows us to estimate not only the Gowers norms of the modified von Mangoldt function, but also that of the product of the modified von Mangoldt function and some "well-behaved" sequences. The method we use basically follows the
one used in ~\cite{5}. We prove:
\begin{prop}\label{g}
(1)
For any bounded basic k-step nilsequence $\{F(g^{n}x)\}_{n\in\mathbb{N}}$ with bounded Lipschitz norm, both
\begin{equation}\nonumber
\begin{split}
\sup_{1\leq r<w, (r,w)=1}\Bigl\Vert(\Lambda'_{w, r}-1)\bold{1}_{[1, N]}(n)F(g^{Wn}x)\Bigr\Vert_{{U^{k}(\mathbb{Z}_{kN})}}
\end{split}
\end{equation}
and
\begin{equation}\nonumber
\begin{split}
\sup_{1\leq r<w, (r,w)=1}\Bigl\Vert(\Lambda'_{w, r}-1)\bold{1}_{[1, N]}(n)F(g^{n}x)\Bigr\Vert_{{U^{k}(\mathbb{Z}_{kN})}}
\end{split}
\end{equation}
converge to 0 as $N\rightarrow\infty$ and then $w\rightarrow\infty$.

(2)Let $(\mathbb{T},\mathcal{X},\lambda,R_{\alpha})$ be the 1-dimensional torus with $\alpha\in\mathbb{R}\backslash\mathbb{Q}$.
Let $A\in\mathcal{X}$ be an interval. Then for any $x\in\mathbb{T}$, both
\begin{equation}\nonumber
\begin{split}
\sup_{1\leq r<w, (r,w)=1}\Bigl\Vert(\Lambda'_{w, r}-1)\bold{1}_{[1, N]}(n)\bold{1}_{A}(R_{\alpha}^{Wn}x)\Bigr\Vert_{{U^{k}(\mathbb{Z}_{kN})}}
\end{split}
\end{equation}
and
\begin{equation}\nonumber
\begin{split}
\sup_{1\leq r<w, (r,w)=1}\Bigl\Vert(\Lambda'_{w, r}-1)\bold{1}_{[1, N]}(n)\bold{1}_{A}(R_{\alpha}^{n}x)\Bigr\Vert_{{U^{k}(\mathbb{Z}_{kN})}}
\end{split}
\end{equation}
converge to 0 as $N\rightarrow\infty$ and then $w\rightarrow\infty$.
\end{prop}
\begin{rem}
  Our definition of nilsequences is different from the one used in ~\cite{5}. The basic nilsequence defined in this paper is called a nilsequence in ~\cite{5}.
\end{rem}
\begin{proof} For any $\delta>0$, we can find two non-negative smooth functions $F_{l}^{\delta}(x)\leq \bold{1}_{A}(x)\leq F_{u}^{\delta}(x)$ such that $\int_{\mathbb{T}}(F_{u}^{\delta}-F_{l}^{\delta})d\lambda<\delta/4$.
We temporally write $G_{W}(n)$ to represent one of the following functions: $F(g^{Wn}x),F(g^{n}x),\bold{1}_{A}(R_{\alpha}^{Wn}x),\bold{1}_{A}(R_{\alpha}^{n}x)$. Suppose the statement is not true for one of the functions $G_{W}(n)$ listed above. Then there exists $\delta>0$ and sequences of non-negative integers $\{w_{N}\}_{N\in\mathbb{N}},\{r_{N}\}_{N\in\mathbb{N}}, \{W_{N}\}_{N\in\mathbb{N}}$ and  $I\subset\mathbb{N}$ an infinite subset, such
that $W_{N}=\prod_{p\in\mathbb{P}, p<w_{N}}p, W_{N}=o(\log N), r_{N}$ is coprime with $W_{N}$, and $W_{N}\rightarrow\infty$ as $N\rightarrow\infty$, and
\begin{equation}\nonumber
\begin{split}
\Bigl\Vert(\Lambda'_{w_{N}, r_{N}}-1)\bold{1}_{[1, N]}(n)G_{W_{N}}(n)\Bigr\Vert_{{U^{k}(\mathbb{Z}_{kN})}}>\delta
\end{split}
\end{equation}
for all $N\in I$. Moreover, $I$ can be chosen such that
\begin{equation}\label{t1}
\begin{split}
\Bigl\vert\frac{1}{kN}\sum_{n=1}^{kN}F_{c}^{\delta'}(R_{\alpha}^{W_{N}n}x)-\int_{\mathbb{T}} F_{c}^{\delta'}d\lambda\Bigr\vert<\delta'/8
\end{split}
\end{equation}
and
\begin{equation}\label{t2}
\begin{split}
\Bigl\vert\frac{1}{kN}\sum_{n=1}^{kN}F_{c}^{\delta'}(R_{\alpha}^{n}x)-\int_{\mathbb{T}} F_{c}^{\delta'}d\lambda\Bigr\vert<\delta'/8
\end{split}
\end{equation}
for $c=l$, and $u, N\in I$(we can do so because the order of the choice is that first a sequence of $w$ is picked and then we attach to each $w$ with some $N$ sufficiently large). Here $\delta'$ is some constant to be chosen latter depending only on $\delta$ and $k$. By Propositions 10.1 and 6.4 of ~\cite{5}, (Proposition 10.1 of ~\cite{5} was based on the inverse
conjecture for the Gowers norms which was later proved in ~\cite{6}) there exists $\delta'>0$ and a finite collection of $(k-1)$-step nilmanifolds
$\mathcal{U}$ such that for any$N\in I$, there exists a basic $(k-1)$-step nilsequence $\{F'_{N}(h^{n}x)\}_{n\in\mathbb{Z}}$ from one of the
manifolds in $\mathcal{U}$ with bound 1 and Lipschitz constant $O_{\delta,k}(1)$ such that
\begin{equation}\label{mo}
\begin{split}
\Bigl\vert\mathbb{E}_{n\leq kN}(\Lambda'_{w_{N}, r_{N}}-1)\bold{1}_{[1, N]}(n)G_{W_{N}}(n)F'_{N}(h^{n}x)\Bigr\vert>\delta'.
\end{split}
\end{equation}
We now choose the $\delta'$ appearing in (\ref{t1}) and (\ref{t2}) to be the same $\delta'$ appearing in (\ref{mo}) (this is possible because $\delta'$ is independent of $N$).
By passing to an infinite subset of $I$, we can assume all $\{F'_{N}(h^{n}x)\}_{n\in\mathbb{N}}$ comes from the same nilmanifold $Y$. We still denote this subset by $I$ for convenience.

\textbf{Case that $G_{W}(n)=F(g^{Wn}x)$ or $F(g^{n}x)$.} In this case, for each $N\in I, \{S_{N}(n)\colon=G_{W_{N}}(n)F'_{N}(h^{n}x)\}_{n\in\mathbb{N}}$ is a basic nilsequence from $X\times Y$ with a uniform Lipschitz bound.
Then (\ref{mo}) contradicts Proposition 10.2 of ~\cite{5}.

\textbf{Case that $G_{W}(n)=\bold{1}_{A}(R_{\alpha}^{n}x)$ or $\bold{1}_{A}(R_{\alpha}^{Wn}x)$.} We assume $G_{W}(n)=\bold{1}_{A}(R_{\alpha}^{Wn}x)$ since the proof of the other case is identical.
We may assume without loss of generality that all functions $F'_{N}(x), N\in I$ are nonnegative (otherwise we can split $F'_{N}(x)$ as the difference of two
continuous nonnegative functions and use the argument for each one). Then
\begin{equation}\label{r2}
\begin{split}
&\qquad\mathbb{E}_{n\leq kN}(\Lambda'_{w_{N}, r_{N}}-1)\bold{1}_{[1, N]}(n)\bold{1}_{A}(R_{\alpha}^{W_{N}n}x)F'_{N}(h^{n}x)
\\&\leq\mathbb{E}_{n\leq kN}(\Lambda'_{w_{N}, r_{N}}(n)F_{u}^{\delta'}(R_{\alpha}^{W_{N}n}x)-F_{l}^{\delta'}(R_{\alpha}^{W_{N}n}x))\bold{1}_{[1, N]}(n)F'_{N}(h^{n}x)
\\&\leq\mathbb{E}_{n\leq kN}(\Lambda'_{w_{N},r_{N}}-1)\bold{1}_{[N]}(n)F_{u}^{\delta'}(R_{\alpha}^{W_{N}n}x)F'_{N}(h^{n}x)
\\&\qquad+\mathbb{E}_{n\leq kN}(F_{u}^{\delta'}(R_{\alpha}^{W_{N}n}x)-F_{l}^{\delta'}(R_{\alpha}^{W_{N}n}x)).
\end{split}
\end{equation}
Similar to the discussion of the first case, by Proposition 10.2 of ~\cite{5}, the first term on the right hand side goes to 0 as $N\rightarrow\infty$. For the second term, by (\ref{t1}) we have
\begin{equation}\nonumber
\begin{split}
&\qquad \mathbb{E}_{n\leq kN}(F_{u}^{\delta'}(R_{\alpha}^{W_{N}n}x)-F_{l}^{\delta'}(R_{\alpha}^{W_{N}n}x))
\\&\leq \int_{\mathbb{T}}(F_{u}^{\delta'}-F_{l}^{\delta'})d\lambda+\delta'/4
\\&\leq\delta'/2.
\end{split}
\end{equation}
Thus the left hand side of (\ref{r2}) is less than $\delta'$. Similarly, the left hand side of (\ref{r2}) is greater than $-\delta'$.
This contradicts (\ref{mo}).
\end{proof}

\section{Convergence and Recurrence along Shifted Primes on $[n\alpha]$}

\subsection{Comparison with Averages along Integers}
The following proposition shows that the weighted averages along some properly chosen subsets of $[N]$ for $[n\alpha]$ are controlled by the Gowers norms
of a related function:

\begin{prop}\label{norm4alpha} Let $\alpha\in\mathbb{R}\backslash\mathbb{Q}$. Let $a_{i}(n)=[in\alpha]$ for all $1\leq i\leq k, n\in [N]$, or $a_{i}(n)=i[n\alpha]$ for all $1\leq i\leq k, n\in [N]$. Let $(X,\mathcal{X},\mu, T)$ be a
measure preserving system and $f_{0}, f_{1},\dots,f_{k}\in L^{\infty}(\mu)$. Let $b\colon  \mathbb{N}\rightarrow\mathbb{C}$ be a sequence of complex numbers
satisfying $b(n)/{n^{c}}\rightarrow 0$ for all $c>0$. Let $\xi\colon
\mathbb{T}\rightarrow\mathbb{R}$ be a function with absolute value bounded by 1 whose support is contained in an interval that does not contain any internal point of the form $\frac{i}{(k+1)!},0\leq i< (k+1)!$. Then
\begin{equation}\label{specialcode4}
\begin{split}
&\qquad\Bigl\Vert\frac{1}{N}\sum_{n=1}^{N}b(n)\xi(\{n\alpha\})f_{0}T^{a_{1}(n)}f_{1}\cdot\ldots\cdot T^{a_{k}(n)}f_{k}\Bigr\Vert_{L^{2}(\mu)}
\\&\ll\Bigl\Vert b(n)\xi(\{n\alpha\})\cdot\bold{1}_{[1, N]}\Bigr\Vert_{U^{k+1}(\mathbb{Z}_{(k+1)N})}+o_{N}(1).
\end{split}
\end{equation}
Furthermore, the implicit constant is independent of
$\{b(n)\}_{n\in\mathbb{N}}$ and the $o_{N}(1)$ term depends only on the integer $k$ and $\{b(n)\}_{n\in\mathbb{N}}$.
\end{prop}
\begin{rem}
The function $i[n\alpha]$ is generally easier to treat than $[in\alpha]$, and we can replace $(k+1)!$ with $2$ in the first case. But we write them
together since the proofs are similar.
\end{rem}

\begin{proof} Without loss of generality, we assume that all functions $f_{i}$ are bounded by 1. Before we prove the general case, we give an example for the case $k=1$ which is easier and illustrates the main idea.

By Lemma \ref{van}, the
Cauchy-Schwartz inequality and the invariance of $T$, we have
\begin{equation}\label{temp01}
\begin{split}
&\qquad \Bigl\Vert\frac{1}{N}\sum_{n=1}^{N}b(n)\xi(\{n\alpha\}){f_{0}}{T^{[n\alpha]}f_{1}}\Bigr\Vert^{4}_{L^{2}(\mu)}
\\&\ll\Bigl(\frac{1}{N}\sum_{h=1}^{N}\Bigl\Vert\frac{1}{N}\sum_{n=1}^{N-h}b(n)\xi(\{n\alpha\})\overline{b(n+h)\xi(\{(n+h)\alpha\})}T^{[n\alpha]}f_{0}\overline{f_{0}}f_{1}T^{[(n+h)\alpha]}\overline{f_{1}}\Bigr\Vert_{L^{1}(\mu)}\Bigr)^{2}+o_{N}(1)
\\&\ll\frac{1}{N}\sum_{h=1}^{N}\Bigl\Vert\frac{1}{N}\sum_{n=1}^{N-h}b(n)\xi(\{n\alpha\})\overline{b(n+h)\xi(\{(n+h)\alpha\})}T^{[n\alpha]}f_{0}\overline{f_{0}}f_{1}T^{[(n+h)\alpha]}\overline{f_{1}}\Bigr\Vert_{L^{1}(\mu)}^{2}+o_{N}(1)
\\&\ll\frac{1}{N}\sum_{h=1}^{N}\Bigl\Vert\frac{1}{N}\sum_{n=1}^{N-h}b(n)\xi(\{n\alpha\})\overline{b(n+h)\xi(\{(n+h)\alpha\})}f_{1}T^{[(n+h)\alpha]-[n\alpha]}\overline{f_{1}}\Bigr\Vert^{2}_{L^{2}(\mu)}+o_{N}(1),
\end{split}
\end{equation}
where $o_{N}(1)$ depends on $\{b(n)\}_{n\in\mathbb{N}}$. Suppose $\xi(x)$ is supported in an interval $I\subset(0,1), \vert I\vert\leq 1/2$. Notice that
$[(n+h)\alpha]-[n\alpha]=h\alpha-\{(n+h)\alpha\}+\{n\alpha\}$. Thus if $\{n\alpha\}, \{(n+h)\alpha\}\in I$, then $\{h\alpha\}<1/2$ implies $
[(n+h)\alpha]-[n\alpha]=[h\alpha],$ and $\{h\alpha\}>1/2 $ implies $ [(n+h)\alpha]-[n\alpha]=[h\alpha]+1$. Let
$A_{N}=\{h\in[N]\colon  \{h\alpha\}<1/2\}$. Then by splitting the summand of $h$ along $[N]$ into $A_{N}$ and $A_{N}^{c}$, we get
\begin{equation}\nonumber
\begin{split}
&\qquad
\frac{1}{N}\sum_{h=1}^{N}\Bigl\Vert\frac{1}{N}\sum_{n=1}^{N-h}b(n)\xi(\{n\alpha\})\overline{b(n+h)\xi(\{(n+h)\alpha\})}f_{1}T^{[(n+h)\alpha]-[n\alpha]}\overline{f_{1}}\Bigr\Vert^{2}_{L^{2}(\mu)}
\\&\leq\frac{1}{N}\sum_{h\in A_{N}}\Bigl\vert\frac{1}{N}\sum_{n=1}^{N-h}b(n)\xi(\{n\alpha\})\overline{b(n+h)\xi(\{(n+h)\alpha\})}\Bigr\vert^{2}
\\&\qquad +\frac{1}{N}\sum_{h\in A_{N}^{c}}\Bigl\vert\frac{1}{N}\sum_{n=1}^{N-h}b(n)\xi(\{n\alpha\})\overline{b(n+h)\xi(\{(n+h)\alpha\})}\Bigr\vert^{2}
\\&\ll\Bigl\Vert b(n)\xi(\{n\alpha\})\cdot\bold{1}_{[1, N]}\Bigr\Vert_{U^{2}(\mathbb{Z}_{2N})}^{4}.
\end{split}
\end{equation}
This finishes the proof of the case $k=1$.

We now give the proof of the general case. Write $a_{0}(n)=0$.
Denote $g(n)=b(n)\xi(\{n\alpha\})$ and
$g(n;h_{1},\dots,h_{t})=\prod_{\epsilon\in\{0,1\}^{t}}\mathcal{C}^{\vert\epsilon\vert}g(n+\sum_{j=1}^{t}\epsilon_{j}h_{j})$(Recall that $\mathcal{C}f(n)=\overline{f(n)}, \vert\epsilon\vert=\sum_{j=1}^{t}\epsilon_{j}$).

\textbf{Claim.} For any $0\leq t\leq k$, there exists a function $F(\cdot;h_{1},\dots, h_{t})\in L^{\infty}(\mu)$ bounded by 1 (this function depends
on $h_{1},\dots, h_{t}$ but not on $n$) such that
\begin{equation}\nonumber
\begin{split}
&\qquad\Bigl\Vert\frac{1}{N}\sum_{n=1}^{N}g(n)f_{0}T^{a_{1}(n)}f_{1}\cdot\ldots\cdot T^{a_{k}(n)}f_{k}\Bigr\Vert^{2^{k+1}}_{L^{2}(\mu)}
\\&\ll
\frac{1}{N^{t}}\sum_{1\leq h_{1},\dots,h_{t}\leq N}\Bigl\Vert\frac{1}{N}\sum_{n=1}^{N-h_{1}-\dots-h_{t}}g(n;h_{1},\dots,h_{t})
\\&\qquad F(\cdot;h_{1},\dots,
h_{t})\prod_{i=t+1}^{k}\prod_{\epsilon\in\{0,1\}^{t}}T^{a_{i}(n+\sum_{j=1}^{t}\epsilon_{j}h_{j})-a_{t}(n)}\mathcal{C}^{\vert\epsilon\vert}f_{i}\Bigr\Vert^{2^{k+1-t}}_{L^{2}(\mu)}+o_{N}(1).
\end{split}
\end{equation}

If the claim is true, let $t=k$, we get the left hand side of (\ref{specialcode4}) is bounded by
\begin{equation}\nonumber
\begin{split}
\Bigl(\frac{1}{N^{t}}\sum_{1\leq h_{1},\dots,h_{t}\leq
N}\Bigl\vert\frac{1}{N}\sum_{n=1}^{N-h_{1}-\dots-h_{k}}g(n;h_{1},\dots,h_{k})\Bigr\vert^{2}\Bigr)^{\frac{1}{2^{k+1}}}+o_{N}(1),
\end{split}
\end{equation}
which is exactly the right hand side of (\ref{specialcode4}) since this is the expansion of the Gowers norm $\Vert\cdot\Vert_{U^{k+1}(\mathbb{Z}_{(k+1)N})}$ of $g(n)\bold{1}_{[1, N]}(n)$. In order to prove the claim, it suffices to show that for any $1\leq t\leq k$, any $F'(\cdot;h_{1},\dots, h_{t-1})\in
L^{\infty}(\mu)$ bounded by 1, there exists $F(\cdot;h_{1},\dots, h_{t})\in L^{\infty}(\mu)$ bounded by 1, such that
\begin{equation}\label{specialcode2}
\begin{split}
&\qquad\Bigl\Vert\frac{1}{N}\sum_{n=1}^{N-h_{1}-\dots-h_{t-1}}g(n;h_{1},\dots,h_{t-1})
\\&\qquad F'(\cdot;h_{1},\dots,
h_{t-1})\prod_{i=t}^{k}\prod_{\epsilon\in\{0,1\}^{t-1}}T^{a_{i}(n+\sum_{j=1}^{t-1}\epsilon_{j}h_{j})-a_{t-1}(n)}\mathcal{C}^{\vert\epsilon\vert}f_{i}\Bigr\Vert^{2^{k+2-t}}_{L^{2}(\mu)}
\\&\ll
\frac{1}{N}\sum_{h_{t}=1}^{N}\Bigl\Vert\frac{1}{N}\sum_{n=1}^{N-h_{1}-\dots-h_{t}}g(n;h_{1},\dots,h_{t})
\\&\qquad F(\cdot;h_{1},\dots,
h_{t})\prod_{i=t+1}^{k}\prod_{\epsilon\in\{0,1\}^{t}}T^{a_{i}(n+\sum_{j=1}^{t}\epsilon_{j}h_{j})-a_{t}(n)}\mathcal{C}^{\vert\epsilon\vert}f_{i}\Bigr\Vert^{2^{k+1-t}}_{L^{2}(\mu)}+o_{N}(1),
\end{split}
\end{equation}
for any $1\leq h_{1},\dots,h_{t-1}\leq N$. Similar to (\ref{temp01}), by Lemma \ref{van}, the Cauchy-Schwartz inequality and the invariance of $T$, the left hand side of (\ref{specialcode2}) is bounded by $o_{N}(1)$ plus
\begin{equation}\nonumber
\begin{split}
&\frac{1}{N}\sum_{h_{t}=1}^{N}\Bigl\Vert\frac{1}{N}\sum_{n=1}^{N-h_{1}-\dots-h_{t}}g(n;h_{1},\dots,h_{t})
\\&\prod_{i=t}^{k}\prod_{\epsilon\in\{0,1\}^{t}}T^{a_{i}(n+\sum_{j=1}^{t}\epsilon_{j}h_{j})-a_{t}(n)-a_{t-1}(n+\epsilon_{t}h_{t})+a_{t-1}(n)}
\mathcal{C}^{\vert\epsilon\vert}f_{i}\Bigr\Vert^{2^{k+1-t}}_{L^{2}(\mu)},
\end{split}
\end{equation}
where $o_{N}(1)$ depends only on $t$ and $\vert g(n;h_{1},\dots,h_{t})\vert$ $\{b(n)\}_{n\in\mathbb{N}}$, and thus depends only on $k$ and $\{b(n)\}_{n\in\mathbb{N}}$ since $\xi$ is bounded by 1. So in order to prove the claim, it
suffices to show that if $g(n;h_{1},\dots,h_{t})\neq 0$,
\begin{equation}\nonumber
\begin{split}
 F=\prod_{\epsilon\in\{0,1\}^{t}}T^{a_{t}(n+\sum_{j=1}^{t}\epsilon_{j}h_{j})-a_{t}(n)-a_{t-1}(n+\epsilon_{t}h_{t})+a_{t-1}(n)}\mathcal{C}^{\vert\epsilon\vert}f_{i}
\end{split}
\end{equation}
  is a function bounded by 1 and is independent of the choice of $n$. It suffices to show that for any $\epsilon\in\{0,1\}^{t}$, if
  $g(n;h_{1},\dots,h_{t})\neq 0$, then value of
\begin{equation}\label{specialcode1}
\begin{split}
  a_{v}(n+\sum_{j=1}^{t}\epsilon_{j}h_{j})-a_{v}(n)
\end{split}
\end{equation}
 is independent of $n$, where $1\leq v\leq k$.
Suppose $\xi(x)$ is supported in an interval $I\subset\mathbb{T}$ where $\frac{i}{(k+1)!}$ is not an internal point of $I$ for all $0\leq i<(k+1)!$. We discuss case by case:

\textbf{Case that $a_{v}(n)=v[n\alpha]$.} Notice that
\begin{equation}\nonumber
\begin{split}
 v[(n+\sum_{j=1}^{t}\epsilon_{j}h_{j})\alpha]-v[n\alpha]=
 v\sum_{r=1}^{t}\Bigl([(n+\sum_{j=1}^{r}\epsilon_{j}h_{j})\alpha]-[(n+\sum_{j=1}^{r-1}\epsilon_{j}h_{j})\alpha]\Bigr).
\end{split}
\end{equation}
 If $g(n;h_{1},\dots,h_{t})\neq 0$, then for each $1\leq r\leq t$, we must have $\{(n+\sum_{j=1}^{r}\epsilon_{j}h_{j})\alpha\},\\
 \{(n+\sum_{j=1}^{r-1}\epsilon_{j}h_{j})\alpha\}\in I$. By the property of $I$, the length of $I$ does not exceed $\frac{1}{2}$ and $I$ does not contain 
 0 as an internal point. Thus it is easy to verify that
\begin{equation}\nonumber
\begin{split}
 [(n+\sum_{j=1}^{r}\epsilon_{j}h_{j})\alpha]-[(n+\sum_{j=1}^{r-1}\epsilon_{j}h_{j})\alpha]=\epsilon_{r}([h_{r}\alpha]+D(h_{r})),
\end{split}
\end{equation}
where $D(h_{r})$ is the unique closest integer 
to $\{h_{r}\alpha\}$. This means $(\ref{specialcode1})$ is
independent of $n$ if $g(n;h_{1},\dots,h_{t})\neq 0$.

\textbf{Case that $a_{v}(n)=[vn\alpha]$.} Notice that
\begin{equation}\nonumber
\begin{split}
 [v(n+\sum_{j=1}^{t}\epsilon_{j}h_{j})\alpha]-[vn\alpha]=
 \sum_{r=1}^{t}\Bigl([v(n+\sum_{j=1}^{r}\epsilon_{j}h_{j})\alpha]-[v(n+\sum_{j=1}^{r-1}\epsilon_{j}h_{j})\alpha]\Bigr).
\end{split}
\end{equation}
 If $g(n;h_{1},\dots,h_{t})\neq 0$, then for each $1\leq r\leq t$, we must have $\{(n+\sum_{j=1}^{r}\epsilon_{j}h_{j})\alpha\},\\
 \{(n+\sum_{j=1}^{r-1}\epsilon_{j}h_{j})\alpha\}\in I$, so $\{v((n+\sum_{j=1}^{r}\epsilon_{j}h_{j}))\alpha\},
 \{v(n+\sum_{j=1}^{r-1}\epsilon_{j}h_{j})\alpha\}\in I_{v}\colon=\{vx\colon x\in I\}$. By the property of $I$, the length of $I_{v}$ does not exceed $\frac{1}{v+1}$ and
 $I_{v}$ does not contain 
 0 as an internal point. Thus it is easy to verify that
\begin{equation}\nonumber
\begin{split}
 [v(n+\sum_{j=1}^{r}\epsilon_{j}h_{j})\alpha]-[v(n+\sum_{j=1}^{r-1}\epsilon_{j}h_{j})\alpha]=\epsilon_{r}([vh_{r}\alpha]+D_{v}(h_{r})),
\end{split}
\end{equation}
where $D_{v}(h_{r})$ is the unique closest integer 
to $v\{h_{r}\alpha\}$. This means
$(\ref{specialcode1})$ is independent of $n$ if $g(n;h_{1},\dots,h_{t})\neq 0$.
\end{proof}

\subsection{Proof of Theorem \ref{alpha2}} We prove Theorem \ref{alpha2} in this subsection.
\begin{defn}
  Let $(X,\mathcal{X},\mu, T)$ be a measure preserving system. A factor $\mathcal{Z}$ is called a \emph{characteristic factor}, or \emph{characteristic}, for the family of integer sequences $\{a_{1}(n),\dots,a_{k}(n)\}$, if for any $f_{1},\dots,f_{k}\in L^{\infty}(\mu)$,at least one of which is orthogonal to $\mathcal{Z}$, the average
\begin{equation}\nonumber
\begin{split}
\frac{1}{N}\sum_{n=1}^{N}T^{a_{1}(n)}f_{1}(x)\cdot\ldots\cdot T^{a_{k}(n)}f_{k}(x)
\end{split}
\end{equation}
converges to 0 in $L^{2}(\mu)$ as $N\rightarrow\infty$.
\end{defn}

The following theorem proved plays an important role in the study of the characteristic factors:
\begin{thm}(Host and Kra, ~\cite{12})\label{st}
  For any $k\in\mathbb{N}$ and any measure preserving system, there exists a factor $\mathcal{Z}_{k}$ which is characteristic for $\{n,2n,\dots,kn\}$, and a.e. ergodic component of $\mathcal{Z}_{k}$ is an inverse limit of $k$-step nilsystems (See Section 2 for the definition).
\end{thm}

We refer the reader to ~\cite{12} for details about characteristic factors. The following proposition can be deduced directly from Proposition 4.1 of ~\cite{2}:
\begin{prop}\label{last}
Let $(X,\mathcal{X},\mu, T)$ be a measure preserving system. Let $\alpha\in\mathbb{R}\backslash\mathbb{Q}, a,b\in\mathbb{Z}$. Then there exists $l\in\mathbb{N}$ such that the factor $\mathcal{Z}_{l}$ in Theorem \ref{st} is characteristic for both $\{[(an+b)\alpha],\dots,k[(an+b)\alpha]\}$ and $\{[(an+b)\alpha],\dots,[k(an+b)\alpha]\}$.
\end{prop}
\begin{rem}
Proposition 4.1 of ~\cite{2} is stated for the case $\{[(an+b)\alpha],\dots,k[(an+b)\alpha]\}$ but the proof of the case $\{[(an+b)\alpha],\dots,[k(an+b)\alpha]\}$ is included in the proof.
\end{rem}
The following result is partially due to Frantzikinakis ~\cite{eq}, Lemma 4.7:
\begin{prop}\label{nil}
Let $\alpha\in\mathbb{R}\backslash\mathbb{Q}, a,b\in\mathbb{Z}$. For any nilmanifold $X=G/\Gamma$, with $G$ connected and simply connected, any $g\in G, x_{0}\in X$ and any $F_{1}(x),\dots,F_{k}(x)\in C(X)$, the limit
\begin{equation}\nonumber
\begin{split}
 \lim_{N\rightarrow\infty}\frac{1}{N}\sum_{n=1}^{N}F_{1}(g^{[(an+b)\alpha]}x_{0})\cdot\ldots\cdot F_{k}(g^{k[(an+b)\alpha]}x_{0})
\end{split}
\end{equation}
exists. The conclusion still holds if $i[n\alpha]$ is replaced with $[in\alpha]$ for $i=1,\dots,k$.
\end{prop}

\begin{proof} The case $a=0$ is trivial, so in the rest of the proof we assume $a\neq 0$. Without loss of generality, we assume $x_{0}=\Gamma$.
  The case $\{[n\alpha],\dots,k[n\alpha]\}$ is Lemma 4.7 of ~\cite{eq}. So we only need to prove for the case $\{[n\alpha],\dots,[kn\alpha]\}$. Let
  $\hat{X}=\hat{G}/\hat{\Gamma}$, where $\hat{G}=\mathbb{R}^{k}\times G^{k}, \hat{\Gamma}=\mathbb{Z}^{k}\times\Gamma^{k}$ and $\hat{g}=(1,\dots,1,g,g^{2},\dots,g^{k})$. Then $\hat{G}$ is also connected and simply connected. It is well known that if $G$ is connected and simply connected, the exponential map $\exp(x)\colon\mathfrak{g}\rightarrow G$ is a bijection, where $\mathfrak{g}$ is the Lie algebra of $G$. So for any $t\in\mathbb{R}$, we may define $g^{t}=\exp(tA)$, where $A\in\mathfrak{g}$ is the unique element such that $g=\exp(A)$. We define a function $\hat{F}\colon\hat{X}\rightarrow\mathbb{C}$ by
\begin{equation}\nonumber
\begin{split}
 \hat{F}(t_{1}\mathbb{Z},\dots,t_{k}\mathbb{Z},g_{1}\Gamma,\dots,g_{k}\Gamma)=F_{1}(g^{-\{t_{1}\}}g_{1}\Gamma)\cdot F_{2}(g^{-\{2t_{2}\}}g_{2}\Gamma)\cdot\ldots\cdot F_{k}(g^{-\{kt_{k}\}}g_{k}\Gamma).
\end{split}
\end{equation}
(We caution the reader that $\hat{F}$ may not be continuous.) Notice that
\begin{equation}\nonumber
\begin{split}
 &\qquad\hat{F}(\hat{g}^{(an+b)\alpha}\hat{\Gamma})
 =F_{1}(g^{-\{(an+b)\alpha\}}g^{(an+b)\alpha}\Gamma)\cdot\ldots\cdot F_{k}(g^{-\{k(an+b)\alpha\}}g^{k(an+b)\alpha}\Gamma)
 \\&=F_{1}(g^{[(an+b)\alpha]}\Gamma)\cdot\ldots\cdot F_{k}(g^{[k(an+b)\alpha]}\Gamma),
\end{split}
\end{equation}
so it suffices to show that the limit
\begin{equation}\nonumber
\begin{split}
 \lim_{N\rightarrow\infty}\frac{1}{N}\sum_{n=1}^{N}\hat{F}(\hat{g}^{(an+b)\alpha}\hat{\Gamma})
\end{split}
\end{equation}
exists. For any $\delta>0$ (and sufficiently small) there exists a function $\hat{F}_{\delta}\in C(\hat{X})$ that agrees with $\hat{F}$ on $\hat{X}_{\delta}=I_{\delta}\times \hat{X}$, where $I_{\delta}=\{(t_{1}\mathbb{Z},\dots,t_{k}\mathbb{Z})\colon \min_{0\leq j<i-1}\bigl\vert t_{i}-\frac{j}{i}\bigr\vert\geq\delta, i=1,\dots,k\}$, and is uniformly bounded by $2\Vert \hat{F}\Vert_{\infty}$. Denote $I_{0}=\{t\mathbb{Z}: \min_{0\leq j<i\leq k}\bigl\vert t-\frac{j}{i}\bigr\vert\geq \delta\}$. If $t\in I_{0}$, then $(t,\dots,t)\in I_{\delta}$. Since $(\{(an+b)\alpha\})_{n\in\mathbb{N}}$ is uniformly distributed on $\mathbb{T}$, the density of the set of $n$ in $\{1,\dots,N\}$ (as $N\rightarrow\infty$) such that $\hat{F}(\hat{g}^{(an+b)\alpha}\hat{\Gamma})\neq\hat{F}_{\delta}(\hat{g}^{(an+b)\alpha}\hat{\Gamma})$ is at most $1-\vert I_{0}\vert=C_{k}\delta$ for some $C_{k}>0$ depending on $k$. So
\begin{equation}\label{temp3}
\begin{split}
 \limsup_{N\rightarrow\infty}\frac{1}{N}\sum_{n=1}^{N}\Bigl\vert\hat{F}(\hat{g}^{(an+b)\alpha}\hat{\Gamma})-\hat{F}_{\delta}
 (\hat{g}^{(an+b)\alpha}\hat{\Gamma})\Bigr\vert\leq 4C_{k}\Vert\hat{F}\Vert_{\infty}\delta.
\end{split}
\end{equation}
By Theorem B of ~\cite{L2}, the limit
\begin{equation}\nonumber
\begin{split}
 \lim_{N\rightarrow\infty}\frac{1}{N}\sum_{n=1}^{N}\hat{F}_{\delta}(\hat{g}^{(an+b)\alpha}\hat{\Gamma})
\end{split}
\end{equation}
exists. It is then easy to deduce from (\ref{temp3}) that $\{\frac{1}{N}\sum_{n=1}^{N}\hat{F}(\hat{g}^{(an+b)\alpha}\hat{\Gamma})\}_{N\in\mathbb{N}}$ is a Cauchy sequence. So the limit exists.
\end{proof}

With the help of all the above material, we are now able to prove the following multiple convergence result:
\begin{prop}\label{or4alpha2} Let $\alpha\in\mathbb{R}\backslash\mathbb{Q}, k\in\mathbb{N}$. For any measure preserving system $(X,\mathcal{X},\mu, T)$ and any
$f_{1},\dots,f_{k}\in L^{\infty}(\mu)$, the average
\begin{equation}\nonumber
\begin{split}
\frac{1}{N}\sum_{n=1}^{N}T^{[(an+b)\alpha]}f_{1}(x)\cdot\ldots\cdot T^{k[(an+b)\alpha]}f_{k}(x)
\end{split}
\end{equation}
converges in $L^{2}(\mu)$ as $N\rightarrow\infty$ for all integers $a, b$. The conclusion also holds if $i[(an+b)\alpha]$ is replaced with $[i(an+b)\alpha]$ for $i=1,\dots,k$.
\end{prop}
\begin{rem}
  It is worth noting that the case $\{[n\alpha],\dots,k[n\alpha]\}$ is proved in Theorem 2.1 of ~\cite{2}.
\end{rem}

\begin{proof} The method follows from ~\cite{eq} and ~\cite{2}.
We can assume $f_{i}, 1\leq i\leq k$ are continuous by an approximation argument. By Proposition \ref{last}, we can replace the original space with $\mathcal{Z}_{l}$ for some $l\in\mathbb{N}$. Using an ergodic decomposition argument, it suffices to prove this proposition when the system is an inverse limit of nilsystems. By an approximation argument, we can simply consider the case when the system $X=G/\Gamma$ is a nilsystem.

We further simplify our discussion as follows (see ~\cite{L2}): since the average we deal with involves only finitely many actions on $X$, we may assume the discrete group $G/G_{0}$ is finitely generated, where $G_{0}$ is the connected component of $G$ containing the unit element $id_{G}$. In this case one can show that $X$ is isomorphic to a sub-nilmanifold of a nilmanifold $\hat{X}=\hat{G}/\hat{\Gamma}$, where $\hat{G}$ is connected and simply connected and for any $F\in C(X), b\in G, x_{0}\in X$, there exists $\hat{F}\in C(\hat{X}), \hat{b}\in \hat{G}, \hat{x}_{0}\in \hat{X}$ such that $F(b^{n}x_{0})=\hat{F}(\hat{b}^{n}\hat{x}_{0})$ for any $n\in\mathbb{N}$. So it suffices to consider the case when $G$ is connected and simply connected, and the conclusion follows from Proposition \ref{nil}.
\end{proof}

\begin{proof}[Proof of Theorem \ref{alpha2}] We only prove for the case $a_{i}(n)=i[n\alpha]$ since the other case follows analogously.
 By Lemma \ref{lambda2}, it suffices to prove the convergence in $L^{2}(\mu)$ for the corresponding averages
\begin{equation}\nonumber
\begin{split}
A(N)=\frac{1}{N}\sum_{n=1}^{N}\Lambda'(n)T^{[n\alpha]}f_{1}\cdot\ldots\cdot T^{k[n\alpha]}f_{k}.
\end{split}
\end{equation}
Equivalently, it suffices to show that the sequence of functions $\{A(N)\}_{N\in\mathbb{N}}$ is Cauchy in $L^{2}(\mu)$. Let $\epsilon>0$. Fix $w, r\in\mathbb{N}$ and let
\begin{equation}\nonumber
\begin{split}
B_{w,r}(N)=\frac{1}{N}\sum_{n=1}^{N}T^{[(Wn+r)\alpha]}f_{1}\cdot\ldots\cdot T^{k[(Wn+r)n\alpha]}f_{k}.
\end{split}
\end{equation}
Let
\begin{equation}\nonumber
\begin{split}
&A_{i}(N)=\frac{1}{N}\sum_{n=1}^{N}\Lambda'(n)\xi_{i}(\{n\alpha\})T^{[n\alpha]}f_{1}\dots T^{k[n\alpha]}f_{k},
\\&A_{i,w,r}(N)=\frac{1}{N}\sum_{n=1}^{N}\Lambda_{w,r}'(n)\xi_{i}(\{(Wn+r)\alpha\})T^{[(Wn+r)\alpha]}f_{1}\dots T^{k[(Wn+r)\alpha]}f_{k},
\\&B_{i,w,r}(N)=\frac{1}{N}\sum_{n=1}^{N}\xi_{i}(\{(Wn+r)\alpha\})T^{[(Wn+r)\alpha]}f_{1}\dots T^{k[(Wn+r)n\alpha]}f_{k},
\end{split}
\end{equation}
where $\xi_{i}(x)=\bold{1}_{(\frac{i-1}{(k+1)!},\frac{i}{(k+1)!})}(x)$.
By Propositions \ref{g} and \ref{norm4alpha}, we have that for any $w_{0}\in\mathbb{N}$ (and corresponding $W_{0}\in\mathbb{N}$) large enough and
any $1\leq r\leq W_{0},(r,W_{0})=1$, if $N$ is large enough, then
\begin{equation}\nonumber
\begin{split}
\Bigl\Vert A_{i,w_{0},r}(N)-B_{i,w_{0},r}(N)\Bigr\Vert_{L^{2}(\mu)}\leq\frac{\epsilon}{(k+1)!}.
\end{split}
\end{equation}
Since $A_{i}(W_{0}N)=\frac{1}{\phi(W_{0})}\sum_{1\leq r\leq W_{0},(r,W_{0})=1}A_{i,w_{0},r}(N)$, we deduce
\begin{equation}\nonumber
\begin{split}
\Bigl\Vert A_{i}(W_{0}N)-\frac{1}{\phi(W_{0})}\sum_{1\leq r\leq W_{0},(r,W_{0})=1}B_{i,w_{0},r}(N)\Bigr\Vert_{L^{2}(\mu)}\leq\frac{\epsilon}{(k+1)!}.
\end{split}
\end{equation}
Thus
\begin{equation}\nonumber
\begin{split}
\Bigl\Vert A(W_{0}N)-\frac{1}{\phi(W_{0})}\sum_{1\leq r\leq W_{0},(r,W_{0})=1}B_{w_{0},r}(N)\Bigr\Vert_{L^{2}(\mu)}\leq\epsilon,
\end{split}
\end{equation}
when $N$ is large. By Proposition \ref{or4alpha2}, for $r=1, \dots, W_{0}$, the sequence $\{B_{w_{0},r}(N)\}_{N\in\mathbb{N}}$ converges in
$L^{2}(\mu)$. Therefore, if $N'$ and $N$ are sufficiently large, then for $r=1, \dots, W_{0}$ we have
\begin{equation}\nonumber
\begin{split}
\Bigl\Vert B_{w_{0},r}(N)-B_{w_{0},r}(N')\Bigr\Vert_{L^{2}(\mu)}\leq\epsilon.
\end{split}
\end{equation}
Thus if $N$ is large enough,
\begin{equation}\nonumber
\begin{split}
\Bigl\Vert A(W_{0}N)-A(W_{0}N')\Bigr\Vert_{L^{2}(\mu)}
\leq\Bigl\Vert B_{w_{0},r}(N)-B_{w_{0},r}(N')\Bigr\Vert_{L^{2}(\mu)}+2\epsilon\leq 3\epsilon.
\end{split}
\end{equation}
Notice that for $r=1, \dots, W_{0}$,
\begin{equation}
\begin{split}
\lim_{N\rightarrow\infty}\Bigl\Vert A(W_{0}N+r)-A(W_{0}N)\Bigr\Vert_{L^{2}(\mu)}=0.
\end{split}
\end{equation}
Thus if $N$ and $N'$ are sufficiently large,
\begin{equation}
\begin{split}
\Bigl\Vert A(N)-A(N')\Bigr\Vert_{L^{2}(\mu)}\leq4\epsilon.
\end{split}
\end{equation}
Therefore the sequence $\{A(N)\}_{N\in\mathbb{N}}$ is Cauchy and this finishes the proof.
\end{proof}
\subsection{Proof of Theorem \ref{alpha}}
We prove Theorem \ref{alpha} in this subsection.
The following proposition is from Theorem 2.1 of ~\cite{10}:
\begin{prop}[Bergelson, Host, McCutcheon and Parreau, ~\cite{10}]\label{or4or} For any $\delta>0, k\in\mathbb{N}$, there exists $c(\delta), N(\delta)>0$, such that for any probability system $(X,\mathcal{X},\mu)$, any commuting measure preserving transformations $T_{1},\dots,T_{k}$,
and any $\mu(A)\geq\delta$, there exists $0<n<N(\delta)$ such that
\begin{equation}\nonumber
\begin{split}
\mu(A\cap T_{1}^{-n}A\cap\dots\cap T_{k}^{-n}A)>c(\delta).
\end{split}
\end{equation}
\end{prop}

We need to use the following uniform multiple recurrence result along integers:

\begin{prop}\label{or4alpha} Let $k\in\mathbb{N}$. For any $\delta>0$, there exists $c(\delta)>0$ such that for any $\alpha\in\mathbb{R}\backslash\mathbb{Q}$, any measure
preserving system $(X,\mathcal{X},\mu, T)$ and any $\mu(A)\geq\delta$, we have
\begin{equation}\nonumber
\begin{split}
\liminf_{N\rightarrow\infty}\frac{1}{N}\sum_{n=1}^{N}\mu(A\cap T^{-[n\alpha]}A\cap\dots\cap T^{-k[n\alpha]}A)>c(\delta).
\end{split}
\end{equation}
The conclusion also holds if $i[n\alpha]$ is replaced with $[in\alpha]$ for $i=1,\dots,k$.
\end{prop}
\begin{rem}
  Notice the $c(\delta)$ in this theorem does not depend on $\alpha$.
\end{rem}
\begin{proof}The proof follows the method used in ~\cite{11}.
Denote $q(n)=[n\alpha], p_{m}(n)=[m\alpha]n$. For any $r\in\mathbb{N}$, denote $S_{r}=\{m\in\mathbb{N}\colon  \{m\alpha\}<1/kr\}$. It is easy to see if
$m\in S_{r}$, then $q(inm)=p_{m}(in)=iq(nm), 1\leq n\leq r, 1\leq i\leq k$. Let $c(\delta), N(\delta)>0$ be the same as in Proposition \ref{or4or}
and let $T_{i}=T^{i[m\alpha]}$. We assume without loss of generality that $N(\delta)\in\mathbb{N}$. Then there exists $1\leq n\leq N(\delta)$, such that $\mu(A\cap T^{-p_{m}(n)}A\cap\dots\cap T^{-p_{m}(kn)}A)>c(\delta)$.
Clearly,
\begin{equation}\nonumber
\begin{split}
&\qquad\liminf_{N\rightarrow\infty}\frac{1}{N}\sum_{n=1}^{N}\mu(A\cap T^{-q(n)}A\cap\dots\cap T^{-q(kn)}A)
\\&\geq\liminf_{N\rightarrow\infty}\frac{1}{N}\sum_{m=1}^{[\frac{N}{n}]}\mu(A\cap T^{-q(mn)}A\cap\dots\cap T^{-q(kmn)}A),
\end{split}
\end{equation}
for any $n\in\mathbb{N}$. Notice that the set $S_{N(\delta)}$ has density $1/kN(\delta)$, thus
\begin{equation}\nonumber
\begin{split}
&\qquad\liminf_{N\rightarrow\infty}\frac{1}{N}\sum_{n=1}^{N}\mu(A\cap T^{-q(n)}A\cap\dots\cap T^{-q(kn)}A)
\\&\geq\frac{1}{N(\delta)}\liminf_{N\rightarrow\infty}\frac{1}{N}\sum_{m=1}^{[\frac{N}{N(\delta)}]}\sum_{n=1}^{N(\delta)}\mu(A\cap
T^{-q(mn)}A\cap\dots\cap T^{-q(kmn)}A)
\\&\geq\frac{1}{N(\delta)}\liminf_{N\rightarrow\infty}\frac{1}{N}\sum_{m\in
S_{N(\delta)}\cap\{1,\dots,[\frac{N}{N(\delta)}]\}}\sum_{n=1}^{N(\delta)}\mu(A\cap T^{-p_{m}(n)}A\cap\dots\cap T^{-p_{m}(kn)}A)
\\&\geq \frac{c(\delta)}{kN(\delta)^{3}}>0.
\end{split}
\end{equation}
This finishes the proof of the case $\{[n\alpha],\dots,[kn\alpha]\}$.
The proof of the case
$\{[n\alpha],\dots,\\k[n\alpha]\}$ follows by replacing $q(in)$ with $iq(n)$.
\end{proof}

\begin{proof}[Proof of Theorem \ref{alpha}] We only need to prove the situation for $\mathbb{P}-1$ since the other case is similar.
By Proposition \ref{or4alpha}, there exists a positive constant $c$ depending only on $\mu(A)$ and $k$ such that for any $W>0$
\begin{equation}\nonumber
\begin{split}
&\qquad
\sum_{i=1}^{(k+1)!}\limsup_{N\rightarrow\infty}\frac{1}{N}\sum_{n=1}^{N}\xi_{i}(\{nW\alpha\})\mu(A\cap T^{-[nW\alpha]}A\cap\dots\cap
T^{-[knW\alpha]}A)
\\&=\limsup_{N\rightarrow\infty}\frac{1}{N}\sum_{n=1}^{N}\mu(A\cap T^{-[nW\alpha]}A\cap\dots\cap T^{-[knW\alpha]}A)>c,
\end{split}
\end{equation}
where $\xi_{i}(x)=\bold{1}_{(\frac{i-1}{(k+1)!},\frac{i}{(k+1)!})}(x)$.
So one term on the left hand side is larger than $\frac{c}{
(k+1)!}$. Suppose this is the term corresponding to $ \xi_{i}, i\in\{1,\dots,(k+1)!\}$.

If we apply Proposition \ref{norm4alpha} with $W\alpha, f_{0}=\dots=f_{k}=\bold{1}_{A}, b(n)=\Lambda'_{w,1}(n)-1, \xi(x)=\xi_{i}(x)$ and
use the fact that when $w$ (and corresponding $W$) large enough, $\bigl\Vert (\Lambda'_{w,1}(n)-1)\xi_{i}(\{nW\alpha\})\cdot\bold{1}_{[1, N]}\bigr\Vert_{U^{k}(\mathbb{Z}_{kN})}$ can be sufficiently small as
$N\rightarrow\infty$ (by Proposition \ref{g}), we can deduce that if $w$ is large enough, then
\begin{equation}\nonumber
\begin{split}
\limsup_{N\rightarrow\infty}\frac{1}{N}\sum_{n=1}^{N}\Lambda'_{w,1}(n)\xi_{i}(\{nW\alpha\})\mu(A\cap T^{-[nW\alpha]}A\cap\dots\cap
T^{-[knW\alpha]}A)>0.
\end{split}
\end{equation}
This finishes the proof of the case $\{[n\alpha],\dots,[kn\alpha]\}$ by Lemma \ref{lambda2}. The proof of the case $\{[n\alpha],\dots,k[n\alpha]\}$ can be obtained by replacing $T^{[jnW\alpha]}$ with $T^{j[nW\alpha]}$ in the discussion.
\end{proof}

\section{Other Applications}

\subsection{Cube Averages along Shifted Primes}
Every element $\epsilon\in\{0,1\}^{d}$ is identified as a sequence $\epsilon=(\epsilon_{1},\dots,\epsilon_{d}),
\epsilon_{i}\in\{0,1\}$. Write $\{0,1\}^{d}_{*}=\{0,1\}^{d}\backslash\{(0,\dots,0)\}$. The convergence result along cubes along integers was obtained
in ~\cite{8}. We indicate how the same method used for proving Theorem \ref{alpha} can be applied for the
averages along cubes along primes:

\begin{thm}\label{cube} For any $k\in\mathbb{N}$, any measure preserving system $(X,\mathcal{X},\mu, T)$ and any functions $f_{\epsilon}\in L^{\infty}(\mu),
\epsilon\in\{0,1\}^{k}_{*}$, the average
\begin{equation}\nonumber
\begin{split}
\frac{1}{\pi(N)^{k}}\sum_{p_{1},\dots,p_{k}\in
[N]\cap\mathbb{P}}\prod_{\epsilon\in\{0,1\}^{k}_{*}}T^{p_{1}\epsilon_{1}+\dots+p_{k}\epsilon_{k}}f_{\epsilon}
\end{split}
\end{equation}
converges in $L^{2}(\mu)$ as $N\rightarrow\infty$.
\end{thm}

\begin{thm}\label{cube2} Let $\mathcal{P}=\mathbb{P}+1$ or $\mathbb{P}-1$. For any sequence of integers $\{W_{N}\}_{N\in\mathbb{N}}$ with $\lim_{N\rightarrow\infty}W_{N}=\infty, W_{N}=o(\log N)$, any $k\in\mathbb{N}$, any measure preserving system $(X,\mathcal{X},\mu, T)$ and any
$A\in\mathcal{X}, \mu(A)>0$, we have
\begin{equation}\nonumber
\begin{split}
\liminf_{N\rightarrow\infty}\frac{1}{\pi(N)^{k}}\sum_{n_{1},\dots,n_{k}\in
\mathcal{P}\cap[N]}\mu\Bigl(\bigcap_{\epsilon\in\{0,1\}^{k}}T^{-W_{N}(n_{1}\epsilon_{1}+\dots+n_{k}\epsilon_{k})}A\Bigr)>0.
\end{split}
\end{equation}
\end{thm}
\begin{rem} It is not known whether one can drop the factor $W_{N}$ in this theorem (i.e. $W_{N}=1$). The reason we need $W_{N}$ is that Proposition \ref{g} is the estimate for the modified von Mangoldt function, and we do not have an estimate for the ordinary von Mangoldt function.
\end{rem}
We outline the key ingredients for the proofs.
\begin{prop}\label{norm4cube} Let $k\in\mathbb{N}$, $(X,\mathcal{X},\mu, T)$ be a measure preserving system, $f_{\epsilon}\in L^{\infty}(\mu),
\epsilon\in\{0,1\}^{k}_{*}$ be functions bounded by 1. Let $b_{1},\dots, b_{k}\colon \mathbb{N}\rightarrow\mathbb{C}$ be $k$ sequences of complex
numbers satisfying $b_{i}(n)/{n^{c}}\rightarrow 0$ for all $c>0$ and \\$\frac{1}{N}\sum_{n=1}^{N}\vert b_{i}(n)\vert\leq 1 (1\leq i\leq k)$. Then for any
$1\leq j\leq k, N\in\mathbb{N}$, we have
\begin{equation}\nonumber
\begin{split}
\Bigl\Vert\frac{1}{N^{k}}\sum_{1\leq n_{1},\dots,n_{k}\leq N}\prod_{i=1}^{k}b_{i}(n_{i}){\prod_{\epsilon\in\{0,1\}^{k}_{*}}T^{\epsilon_{1}n_{1}+\dots+\epsilon_{k}n_{k}}f_{\epsilon}}\Bigr\Vert_{L^{2}(\mu)}\ll_{k}\Bigl\Vert
b_{j}\cdot\bold{1}_{[1, N]}\Bigr\Vert_{U^{2}(\mathbb{Z}_{2N})}+o_{N}(1).
\end{split}
\end{equation}
Furthermore, the implicit constant is independent of $\{b_{1}(n)\}_{n\in\mathbb{N}},\dots, \{b_{k}(n)\}_{n\in\mathbb{N}}$ and the $o_{N}(1)$ term depends
only on $\{b_{1}(n)\}_{n\in\mathbb{N}},\dots, \{b_{k}(n)\}_{n\in\mathbb{N}}$.
\end{prop}

\begin{rem}
  Proposition 7.1 of ~\cite{5} shows that this proposition is true for $X=\mathbb{Z}_{N}$.
\end{rem}

\begin{proof} We may assume $j=k$ since it is the same for other cases. Fix $f_{\bold{0}}\in{L^{\infty}(\mu)}$ with norm 1. Denote
\begin{equation}\nonumber
\begin{split}
u(n_{1},\dots,n_{k-1})=\frac{1}{N}\sum_{n_{k}=1}^{N}b_{1}(n_{1})\cdot\ldots\cdot
b_{k}(n_{k}){\prod_{\epsilon\in\{0,1\}^{k}}T^{\epsilon_{1}n_{1}+\dots+\epsilon_{k}n_{k}}f_{\epsilon}}.
\end{split}
\end{equation}
We only need to show that
\begin{equation}\nonumber
\begin{split}
\Bigl\vert\int_{X}\frac{1}{N^{k}}\sum_{1\leq n_{1},\dots,n_{k}\leq N}\prod_{i=1}^{k}b_{i}(n_{i}){\prod_{\epsilon\in\{0,1\}^{k}}T^{\sum_{i=1}^{k}\epsilon_{i}n_{i}}f_{\epsilon}}d\mu\Bigr\vert\ll_{k}\Bigl\Vert b_{k}\cdot\bold{1}_{[1,
N]}\Bigr\Vert_{U^{2}(\mathbb{Z}_{2N})}+o_{N}(1).
\end{split}
\end{equation}
Let $g'_{\epsilon}=f_{\epsilon 0},g_{\epsilon}=f_{\epsilon 1}$. Then
\begin{equation}\nonumber
\begin{split}
&\qquad\Bigl\vert\int_{X}\frac{1}{N^{k}}\sum_{1\leq n_{1},\dots,n_{k}\leq N}\prod_{i=1}^{k}b_{i}(n_{i}){\prod_{\epsilon\in\{0,1\}^{k}}T^{\sum_{i=1}^{k}\epsilon_{i}n_{i}}f_{\epsilon}}d\mu\Bigr\vert
\\&\ll \frac{1}{N^{k-1}}\sum_{1\leq n_{1},\dots,n_{k-1}\leq N}\Bigl\vert\prod_{i=1}^{k-1}b_{i}(n_{i})\Bigr\vert\cdot
\Bigl\Vert\frac{1}{N}\sum_{n=1}^{N}b_{k}(n){\prod_{\epsilon\in\{0,1\}^{k-1}}T^{\sum_{i=1}^{k-1}\epsilon_{i}n_{i}+n}g_{\epsilon}}\Bigr\Vert_{{L^{2}(\mu)}}.
\end{split}
\end{equation}
By Lemma \ref{van} and the invariance of $T$,
\begin{equation}\nonumber
\begin{split}
&\qquad\Bigl\Vert\frac{1}{N}\sum_{n=1}^{N}b_{k}(n){\prod_{\epsilon\in\{0,1\}^{k-1}}T^{\sum_{i=1}^{k-1}\epsilon_{i}n_{i}+n}g_{\epsilon}}\Bigr\Vert_{{L^{2}(\mu)}}^{2}
\\&\ll\frac{1}{N}\sum_{h=1}^{N}\Bigl\vert\frac{1}{N}\sum_{n=1}^{N-h}\int
b_{k}(n)\overline{b_{k}(n+h)}\prod_{\epsilon\in\{0,1\}^{k-1}}T^{\sum_{i=1}^{k-1}\epsilon_{i}n_{i}+n}(g_{\epsilon}T^{h}\overline{g_{\epsilon}})d\mu\Bigr\vert+o_{N}(1)
\\&=\frac{1}{N}\sum_{h=1}^{N}\Bigl\vert\frac{C_{h}}{N}\sum_{n=1}^{N-h}b_{k}(n)\overline{b_{k}(n+h)}\Bigr\vert+o_{N}(1)
\ll\Bigl\Vert b_{k}\cdot\bold{1}_{[1, N]}\Bigr\Vert_{U^{2}(\mathbb{Z}_{2N})}^{2}+o_{N}(1).
\end{split}
\end{equation}
The proposition follows from the fact that
\begin{equation}\nonumber
\begin{split}
\frac{1}{N^{k-1}}\sum_{1\leq n_{1},\dots,n_{k-1}\leq N}\Bigl\vert\prod_{i=1}^{k-1}b_{i}(n_{i})\Bigr\vert
\end{split}
\end{equation}
converges.
\end{proof}

As the detail of the proofs of Theorem \ref{cube} and \ref{cube2} are similar to (in fact, easier than) that of Theorems \ref{alpha2} and \ref{alpha}, we just indicate the main steps. By using Lemma \ref{lambda2} and corresponding convergence and recurrence results along integers (Theorems 1.2 and 1.3 of ~\cite{12}),
it suffices to show that the difference between the averages along primes and that along integers is small. And this is achieved by the upper bound obtained in Proposition \ref{norm4cube} and the estimate of Gowers norm in Proposition \ref{g}. We left the details to the reader.

\subsection{Weighted Averages along Primes}
We indicate
how the same method used for proving Theorem \ref{alpha} can be applied to generalize Theorem 2.24 of ~\cite{8} to primes:
\begin{thm}\label{ww} Let $a(n)$ be a bounded k-step nilsequence. Then for any measure preserving system
$(X,\mathcal{X},\mu, T)$ and any $f_{1},\dots,f_{k}\in L^{\infty}(\mu)$, the average
\begin{equation}\nonumber
\begin{split}
\frac{1}{\pi(N)}\sum_{p\in [N]\cap\mathbb{P}}a(p)T^{p}f_{1}(x)\cdot\ldots\cdot T^{kp}f_{k}(x)
\end{split}
\end{equation}
converges in $L^{2}(\mu)$ as $N\rightarrow\infty$.
\end{thm}

We outline the key ingredients for the proof.

\begin{prop}\label{norm4ww} Let $(X,\mathcal{X},\mu, T)$ be a measure preserving system and $f_{1},\dots,f_{k}\in L^{\infty}(\mu)$. Let $b(n)\colon
\mathbb{N}\rightarrow\mathbb{C}$ be a sequence of complex numbers satisfying $b(n)/{n^{c}}\rightarrow 0, \forall c>0$. Then
\begin{equation}\nonumber
\begin{split}
\Bigl\Vert\frac{1}{N}\sum_{n=1}^{N}b(n)T^{n}f_{1}\cdot\ldots\cdot T^{kn}f_{k}\Bigr\Vert_{L^{2}(\mu)}
\ll\Bigl\Vert b(n)\cdot\bold{1}_{[1, N]}\Bigr\Vert_{U^{k}(\mathbb{Z}_{kN})}+o_{N}(1).
\end{split}
\end{equation}
Furthermore, the implicit constant is independent of $\{b(n)\}_{n\in\mathbb{N}}$ and the $o_{N}(1)$ term depends only on the integer $k$ and
$\{b(n)\}_{n\in\mathbb{N}}$.
\end{prop}
This proposition is proved in ~\cite{4}. The proof is similar and actually easier than the proofs of Propositions \ref{norm4alpha} and
\ref{norm4cube}.

As the detail of the proof of Theorem \ref{ww} is similar to (in fact, easier than) that of Theorem \ref{alpha}, we just indicate the main steps. By using Lemma \ref{lambda2} and the corresponding convergence result along integers (Theorem 2.24 of ~\cite{8}),
it suffices to show that the difference between the averages along primes and that along integers is small. And this is achieved by the upper bound obtained in Proposition \ref{norm4ww} and the estimate of Gowers norm in Proposition \ref{g}. We left the details to the reader.


\end{document}